\documentclass[10pt,a4paper]{article}
\usepackage{amssymb}
\usepackage{amsmath, amsthm, amscd, amsfonts, amssymb, graphicx, color}
\usepackage[bookmarksnumbered, colorlinks, plainpages]{hyperref}
\usepackage{graphicx}
\newtheorem{thm}{Theorem}[section]

\newtheorem{lem}[thm]{Lemma}
\newtheorem{prob}[thm]{Problem}

\newtheorem{ques}[thm]{Question}
\newtheorem{prop}[thm]{Proposition}
\newtheorem{defn}[thm]{Definition}
\newtheorem{rem}[thm]{\bf{Remark}}

\numberwithin{equation}{section}
\numberwithin{equation}{subsection}

\begin{document}

\title{Total dominator coloring of central graphs}

\author{Farshad Kazemnejad$^{1}$, and Adel P. Kazemi$^{2}$\footnote{Corresponding author} \\[1em]
$^{1,2}$ Department of Mathematics, \\ University of Mohaghegh Ardabili, \\ P.O.\ Box 5619911367, Ardabil, Iran. \\
[1em]
$^1$ Email: kazemnejad.farshad@gmail.com \\
$^2$ Email: adelpkazemi@yahoo.com \\[1em]
}

\maketitle

\begin{abstract}
A total dominator coloring of a graph $G$ is a proper coloring of
$G$ in which each vertex of the graph is adjacent to every vertex of
some color class. The total dominator chromatic number of a graph is the minimum number of color classes in a total dominator
coloring of it. Here, we study the total dominator coloring on central graphs by giving some tight bounds for the total dominator chromatic number of the central of a graph, join of two graphs and Nordhaus-Gaddum-like relations. Also we will calculate  the total dominator chromatic number of the central of a path, a cycle, a wheel, a complete graph and a complete multipartite graph.
\\[0.2em]

\\[0.2em]

\end{abstract}

\pagestyle{myheadings}
\markboth{\centerline {\scriptsize  F. Kazemnejad and A. P. Kazemi}}     {\centerline {\scriptsize F. Kazemnejad and A. P. Kazemi,                           \mbox{~~~~~~~~~~~~~~~~~~~~~~~~~~~~~~~~~~~~~~~~~~~~~~~~~~~~~~~~~~~~~~~}Total dominator coloring  of central graphs}}


\section{\bf Introduction}

All graphs considered here are non-empty, finite, undirected and simple. For
standard graph theory terminology not given here we refer to \cite{West}. Let $%
G=(V,E) $ be a graph with the \emph{vertex set} $V$ of \emph{order}
$n(G)$ and the \emph{edge set} $E$ of \emph{size} $m(G)$. The
\emph{open neighborhood} and the \emph{closed neighborhood} of a
vertex $v\in V$ are $N_{G}(v)=\{u\in V\ |\ uv\in E\}$ and
$N_{G}[v]=N_{G}(v)\cup \{v\}$, respectively. The \emph{degree} of a
vertex $v$ is also $deg_G(v)=\vert N_{G}(v) \vert $. The
\emph{minimum} and \emph{maximum degree} of $G$ are denoted by
$\delta =\delta (G)$ and $\Delta =\Delta (G)$, respectively. We write $K_{n}$, $C_{n}$ and $P_{n}$ for a \emph{complete graph}, a \emph{cycle} and a \emph{path} of order $n$, respectively, while $G[S]$, $W_n$ and $K_{n_1,n_2,\cdots,n_p}$ denote the subgraph of $G$ \emph{induced} by a vertex set $S$, a \emph{wheel} of order $n+1$, and a \emph{complete $p$-partite graph}, respectively. The \emph{complement} of a graph $G$, denoted by $\overline{G}$, is a graph with the vertex set $V(G)$ and for every two vertices $v$ and $w$, $vw\in E(\overline{G})$\ if and only if $vw\not\in E(G)$. 

\vskip 0.2 true cm

Vernold et al., in \cite{Vv} by doing an operation on a given graph obtained the central of the graph as following.

\begin{defn}
\emph{\cite{Vv} The} central graph $C(G)$ \emph{of a graph $G=(V,E)$ of order $n$ and size $m$ is a graph of order $n+m$ and size $\binom{n}{2}+m $ which is obtained by subdividing each edge of $G$ exactly once and joining all the non-adjacent vertices of $G$ in $ C(G)$.}
\end{defn}

\vskip 0.2 true cm

\textbf{Total domination number.} Domination in graphs is now well studied in graph theory and the literature
on this subject has been surveyed and detailed in the two books by Haynes,
Hedetniemi, and Slater~\cite{hhs1, hhs2}. A famous type of domination is total domination, and the literature on this subject has been surveyed and detailed in the recent
book~\cite{HeYe13}. A \emph{total dominating set}, briefly TDS, $S$ of a graph $G$ is a subset
of the vertex set of $G$ such that for each vertex $v$, $N_G(v)\cap
S\neq \emptyset$. The \emph{total domination number $\gamma_t(G)$} of $G$ is the minimum cardinality of a TDS of $G$. 

\vskip 0.2 true cm

\textbf{Total dominator Coloring.} A \emph{proper coloring} of a graph $G$ is a function from
the vertices of the graph to a set of colors such that any two
adjacent vertices have different colors. The \emph{chromatic number}
$\chi (G)$ of $G$ is the minimum number of colors needed in a proper
coloring of a graph. In a proper coloring of a graph, a \emph{color class} of the
coloring is a set consisting of all those vertices assigned the same
color. If $f$ is a proper coloring of $G$ with the coloring classes
$V_1$, $V_2$, $\cdots$ , $V_{\ell}$ such that every vertex in $V_i$ has
color $i$, we write simply $f=(V_1,V_2,\cdots,V_{\ell})$. 

\vskip 0.2 true cm

Graph coloring is used as a model for a vast
number of practical problems involving allocation of scarce
resources (e.g., scheduling problems), and has played a key role in
the development of graph theory and, more generally, discrete
mathematics and combinatorial optimization. Graph colorability
is NP-complete in the general case, although the problem is solvable
in polynomial time for many classes \cite{GJ}. 

\vskip 0.2 true cm

Motivated by the relation between coloring and domination the notion of total dominator colorings was introduced in \cite{Kaz2015}, and for more information the reader can study \cite{Hen2015,Kaz2014,Kaz2016,Kaz-Par}.

\begin{defn} 
\label{total dominator coloring} \emph{ \cite{Kaz2015} A} total dominator coloring,
\emph{briefly TDC, of a graph $G$ with a possitive minimum degree is a proper coloring of $G$ in
which each vertex of the graph is adjacent to every vertex of some
color class. The }total dominator chromatic number $\chi_d^t(G)$
\emph{of $G$ is the minimum number of color classes in a TDC of $G$.}
\end{defn}

For a TDC $f=(V_1,V_2,\cdots,V_{\ell})$ of a graph $G$, a vertex $v$ is called a \emph{common neighbor} of
$V_i$ or we say that $v$ \emph{completely dominates} $V_i$ and write $v\succ_t V_i$ if $V_i \subseteq N(v)$. Otherwise we write $v \not\succ_t  V_i$. 
Also a vertex $v$ is called the \emph{private neighbor} of $V_i$ with respect to $f$ if $v\succ_t V_i$ and $v\nsucc_t V_j$ for all $j\neq i$. Also a TDC of $G$  with $\chi_d^t(G)$ colors is called a \emph{min}-TDS. 

\vspace{0.2cm}

The goal of the paper is to study the total dominator chromatic number of central of a graph. In more details, while we give some tight bounds for the total dominator chromatic number of the central of a connected or disconnected graph in Section 2, we discuss on the total dominator chromatic number of the central of the join of two graphs in Section 3. Then after giving some Nordhaus-Gaddum-like relations in Section 4, we will calculate  the total dominator chromatic number of the central of a path, a cycle, a wheel and a complete multipartite graph in the last section.

\vspace{0.2cm}
In this paper, by assumption $V=\{v_1,v_2,\cdots, v_n\}$ as the vertex set of a graph $G$, we consider $V(C(G))=V\cup \mathcal{C}$ as the vertex set of $C(G)$ in which $\mathcal{C}=\{c_{ij}~|~v_iv_j\in E\}$. And so $E(C(G))=\{v_ic_{ij},v_jc_{ij}~|~v_iv_j\in E\}\cup \{v_iv_j~|~v_iv_j\notin E\}$. The following theorems are useful for our investigation.

\begin{thm} \emph{\cite{Kaz2015}}
\label{2=<chi_d ^t=<n} For any connected graph $G$ of order $n$ with $\delta(G)\geq 1$, 
\[
\max\{\chi(G),\gamma_t(G),2\}\leq \chi_d ^t(G)\leq n.
\]
Furthermore, $\chi_d ^t(G)=2$ if and only if $G$ is a
complete bipartite graph, or $\chi_d ^t(G)=n$ if and only if $G$ is a complete graph.
\end{thm}

\begin{thm}\emph{\cite{HeYe13}}
\label{gamm_t (G) =< 4n/7}
If $G\not\in \{C_3,C_5,C_6,C_{10},H_{10},H'_{10}\}$ is a connected graph of order $n$ with $\delta (G) \geq 2$, then $\gamma_t(G)\leq  \lfloor 4n/7\rfloor$.

\end{thm}
\begin{figure}[ht]
\centerline{\includegraphics[width=8.5cm, height=2.8cm]{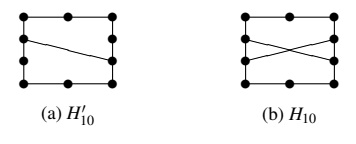}}
\vspace*{-0.3cm}
\caption{The graphs $H_{10}$ and $H'_{10}$}\label{H10}
\end{figure}



\section{\bf General bounds }
In this section, we establish some bounds on the total dominator chromatic number of the central of a graph. First, we consider connected graphs.
\begin{thm}
\label{n/2+2 =< chi_d t(C(G))=<n/2+n}
For any connected graph $G$ of order $n\geq 2$ which its longest path has order $t$,
\[
\lfloor 2n/3 \rfloor +1\leq \chi_d ^t(C(G))\leq n+ \lceil t/2 \rceil.
\]
\end{thm}

\begin{proof}
Let $ G $ be a connected graph of order $ n\geq 2 $ with the vertex set $ V=\{v_{1}, v_{2},\cdots, v_{n}\}$. Then $V(C(G))=V\cup \mathcal{C}$ where $\mathcal{C}=\{ c_{i  j}~|~ v_{i}v_{j}\in E(G) \}$. If $ n=2$, then $G$ is isomorphic to $K_{2}$ and so $C(G)$ is isomorphic to $P_{3}$ and obviously $ \chi_d ^t(C(G))=\lfloor 2n/3\rfloor +1$. So we assume $n\geq 3$. Let $ f=(V_{1}, V_{2},\cdots, V_{\ell} ) $ be a min-TDC of $ C(G)$. Since $ N(c_{ij})=\{v_{i}, v_{j} \},  $ we conclude that if $ c_{ij}\succ_{t} V_{k} $ for some $ k $, then $ V_{k}\subseteq \{v_{i}, v_{j} \}. $ This implies that at least two new colors are needed for coloring of every three vertices, and also at least one new color is needed for coloring the vertices in $\mathcal{C}$. Hence $ \chi^{t}_{d}(C(G))\geq \lfloor 2n/3\rfloor +1$. 
\vspace{0.2cm}

Now let $P_{t}:v_1v_2 \cdots v_t$ be a longest path of order $t$ in $G$, and we consider
\[
V_{1}=\{v_{1},v_{2}\},  ~V_{i}=\{v_{i+1}\} \mbox{ for } 2\leq i \leq n-1,
\]
\[
V_{n-1+k}=\{c_{(2k-1) ( 2k)}\} \mbox{ for } 1\leq k \leq \lfloor t/2\rfloor,
\]
\[
V_{n+\lfloor t/2\rfloor}= \mathcal{C}- (V_1\cup \cdots \cup V_{n+\lfloor t/2\rfloor -1}).
\]
Since the function $ f=(V_{1}, V_{2},\cdots, V_{n+\lfloor t/2 \rfloor} )$ is a TDC of $C(G)$ for even $t$, and the function $ g=(V_{1}, V_{2},\cdots, V_{n+\lfloor t/2 \rfloor-1}, V^{'}_{n+\lfloor t/2 \rfloor}, V^{''}_{n+\lceil t/2 \rceil} )$ is a TDC of $C(G)$ for odd $t$ where
\[
V^{'}_{n+\lfloor t/2 \rfloor}=V_{n+\lfloor t/2 \rfloor}-\{c_{(t-1)t}\},  ~V^{''}_{n+\lceil t/2 \rceil} )=\{c_{(t-1)t}\},
\]
we have $\chi^{t}_{d}(C(G)\leq n+\lceil t/2 \rceil$.
\end{proof}

The following theorem is a trivial result of Theorem \ref{n/2+2 =< chi_d t(C(G))=<n/2+n} for a graph which has a Hamiltonian path.
\begin{thm}
\label{2n/3+1 =< chi_dt C(G) =< n+n/2}
For any graph $G$ of order $n\geq 2$ which has a Hamiltonian path,
\[
\lfloor 2n/3 \rfloor +1\leq \chi_d ^t(C(G))\leq n+ \lceil n/2 \rceil.
\]
\end{thm}
Since for any connected graph $G$ of order $n\geq 2$ and maximum degree at most $n-2$ the coloring function $(\{v_1\}, \dots, \{v_n\},\mathcal{C})$ is a TDC of $C(G)$ where $ V(G)=\{v_{1}, v_{2},\cdots, v_{n}\}$ and $\mathcal{C}=\{ c_{i  j}~|~ v_iv_j\in E(G) \}$, the upper bound $n+\lceil t/2\rceil$ in Theorem \ref{n/2+2 =< chi_d t(C(G))=<n/2+n} can be improved to $n+1$, as we say in Theorem \ref{2n/3 +1 =< chi_d ^t(C(G)) =<n+1}.

\begin{thm}
\label{2n/3 +1 =< chi_d ^t(C(G)) =<n+1}
For any connected graph $G$ of order $n\geq 2$ and maximum degree at most $n-2$,
\[
\lfloor 2n/3 \rfloor+1 \leq \chi_d ^t(C(G))\leq n+1.
\]
\end{thm}
While Theorem \ref{chi_d ^t(C(G))=n+ lceil n/2 rceil} characterizes graphs $G$ which achieve the upper bound $n+ \lceil n/2 \rceil$ in Theorem \ref{2n/3+1 =< chi_dt C(G) =< n+n/2}, Propositions \ref{chi_d^t C(P_n)} and \ref{chi_d^t(C({C_n}))} show the lower bound given in Theorem \ref{2n/3+1 =< chi_dt C(G) =< n+n/2} is also tight. First a lemma.
\begin{lem}
\label{gamma_{t}(C(K_{n}))}
For any integer $ n\geq 2$, $ \gamma_{t}(C(K_{n}))=n+\lceil\dfrac{n}{2}\rceil -1$.
\end{lem}

\begin{proof}
Let $ K_{n} $ be a complete graph of order $ n\geq 2 $ with the vertex  $ V=\{v_{1}, v_{2},\cdots, v_{n}\}$. Then $V\cup \mathcal{C}$ is the partition of the vertex set of the bipartite graph $ C(K_{n}) $ to the independent sets where $\mathcal{C}=\{ c_{i  j}\mid 1\leq i<j \leq n \}$. Let
$S$ be a TDS of $ C(K_{n})$ and let $|S\cap V|=k$ for some $ 1 \leq k \leq n$. Then
\begin{equation*}
\begin{array}{llll}
|\bigcup_{v_i \in S} N(v_i) | & = & (n-1)+(n-2)+\cdots +(n-k) & \mbox{(since } V\cup \mathcal{C} \mbox{ is partition)} \\
                                             & \geq & n(n-1)/2 & \mbox{(since } \mathcal{C} \subseteq \bigcup_{v_i\in S} N(v_i)),\\
\end{array}
\end{equation*}
which implies $k=n-1$. On the other hand, we have $|S \cap V| \geq \lceil n/2 \rceil$ because $V\subseteq \bigcup_{c_{ij}\in S} N(c_{ij})$. Therefore $|S| \geq n+\lceil n/2 \rceil-1$, which implies $ \gamma_{t}(C(K_{n}))\geq n+\lceil n/2\rceil -1$. Now since 
\[
S=\{ v_{i}\mid 1\leq i \leq n-1\} \cup \{ c_{(2i-1)(2i)} |~ 1 \leq i \leq \lceil n/2\rceil\}
\]
is a TDS of $ C(K_{n}) $ with cardinality $ n+\lceil n/2\rceil -1$, we obtain $ \gamma_{t}(C(K_{n}))=n+\lceil n/2 \rceil -1. $
\end{proof}

\begin{thm}
\label{chi_d ^t(C(G))=n+ lceil n/2 rceil}
For any connected graph $G$ of order $n\geq4$, 
\[
\chi_d ^t(C(G))=n+\lceil n/2\rceil \mbox{ if and only if } G\cong K_{n}.
\]
\end{thm}

\begin{proof}
First we prove that for any non-complete graph $G$ of order $ n\geq 4$, $ \chi_d ^t(C(G))<n+\lceil n/2\rceil$. Let $V(G)=\{v_{1}, v_{2},\cdots, v_{n}\}$ and let $v_1v_n \not\in E(G)$. Then $V(C(G))=V(G)\cup \mathcal{C}$ where $\mathcal{C}=\{ c_{i  j}~|~ v_{i}v_{j}\in E(G)\}$). Let $P_{t}:~v_1v_2\cdots v_t$ be a longest path in $G$ of order $t\leq n$. Since $ f=(V_{1}, V_{2},\cdots, V_{n+\lceil t/2 \rceil-1} )$ is a TDC of $C(G)$ where
\[
V_{1}=\{v_{1}\},  ~V_{2}=\{v_{2},v_{3}\}, ~V_{i}=\{v_{i+1}\} \mbox{ for } 3\leq i \leq n-1,
\]
\[
V_{n-1+k}=\{c_{(2k) ( 2k+1)}\} \mbox{ for } 1\leq k \leq \lceil t/2 \rceil-1,
\]
\[
V_{n+\lceil t/2\rceil-1}= \mathcal{C}- (V_1\cup \cdots \cup V_{n+\lceil t/2 \rceil -2}),
\] 
we have $\chi_d ^t(C(G)< n+\lceil n/2\rceil$.
 
 \vspace{0.2cm}

Now we prove $\chi_d ^t(C(K_n))=n+\lceil n/2\rceil$ where $n\geq 4$. Let $ K_{n} $ be a complete graph of order $ n\geq 4$ with the vertex  $ V=\{v_{1}, v_{2},\cdots, v_{n}\}$. Then $V\cup \mathcal{C}$ is the partition of the vertex set of the bipartite graph $ C(K_{n}) $ to the independent sets where $\mathcal{C}=\{ c_{i  j}\mid 1\leq i<j \leq n \}$. 
Now let $f=(V_{1}, V_{2},\cdots,V_{\ell} ) $ be a min-TDC of $ C(K_{n})$. Then $\ell \geq n+\lceil n/2\rceil -1$ by Proposition \ref{2=<chi_d ^t=<n} and Lemma \ref{gamma_{t}(C(K_{n}))}. Let $ \ell= n+\lceil n/2 \rceil -1$. Since $|\{V_{k} ~|~ v_{i}\succ_{t} V_{k} \mbox{ for } 1\leq i \leq n\}|\geq \lceil n/2\rceil $ and so $|\{f(c_{ij}) ~|~ 1\leq i<j \leq n\}|\geq \lceil n/2\rceil +1$, the assumption $ \ell= n+\lceil n/2 \rceil -1$ implies $|\{f(v_{i}) ~|~ 1\leq i \leq n\}|=n-2$. This forces that there exist two color classes $\{v_{i}, v_{j}\} $ and $\{v_{k}, v_{t}\}$ such that $ N(v_{i})\cap  N(v_{k})=\{c_{ik}\}$ and $c_{ik}\nsucc_{t} V_{i} $ for each $ 1\leq i \leq \ell $, a contradiction. Hence $ \ell \geq n+\lceil n/2 \rceil$. Now since $ f=(V_{1}, V_{2},\cdots,V_{n+\lceil n/2 \rceil}) $ is a TDC of $ C(K_{n}) $ where  
\[
V_{i}=\{v_{i}\} \mbox{ for } 1\leq i \leq n-2, ~~  V_{n-1}=\{v_{n-1}, v_{n} \},
\]
\[
V_{n+i}=\{c_{(2i+1) ( 2i+2)}\}~~~~~ \mbox{ for } 0\leq i \leq \lceil n/2\rceil-1,
\]
\[
V_{n+\lceil n/2 \rceil}= V(C(K_{n}))- (V_1\cup \cdots \cup V_{n+\lceil n/2 \rceil -1}),
\]
we have $ \chi_{d}^{t}(C(K_{n}))=n+\lceil n/2\rceil$. 

\vspace{0.2cm}

In Figure \ref{fi:proofexample3}, $(\{v_1\},\{c_{34}\},\{v_3\},\{v_4\},\{c_{12}\},\{c_{56}\},\{v_5,v_6\},\{c_{1i},c_{2i}~|~3\leq i \leq 6\}\cup \{c_{35},c_{36},c_{45},c_{46}\})$ is a min-TDC of $C(K_{6})$.
\end{proof}

\begin{figure}[ht]
\centerline{\includegraphics[width=4.5cm, height=4cm]{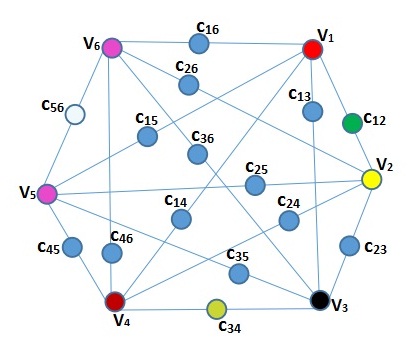}}
\vspace*{-0.3cm}
\caption{A min-TDC of $C(K_{6})$}\label{fi:proofexample3}
\end{figure}

\vspace{0.2cm}

As a remark, obviousely $\chi^{t}_{d}(C({K_n}))=n+\lceil n/2 \rceil -1$ for $n=2,3$ because of
$ C(K_{2})\cong P_{3} $ and $ C(K_{3})\cong C_{6}$.
\vspace{0.2cm}

By reviewing the previous results, Theorem \ref{3=<gama_t(C(G))=< n+n/2-1} can be obtained which improves the upper bound given in Theorem \ref{gamm_t (G) =< 4n/7} \cite{HeYe13} when
\begin{equation*}
m=|E(G)|\geq \left\{
\begin{array}{ll}
(13n-14)/8  & \mbox{for even }n, \\
(13n-7)/8  & \mbox{for odd }n.
\end{array}
\right.
\end{equation*}
\begin{thm}
\label{3=<gama_t(C(G))=< n+n/2-1}
For any connected graph $G$ of order $n\geq 4$,
\[
3\leq \gamma_{t}(C(G))\leq n+\lceil\dfrac{n}{2}\rceil -1.
\]
\end{thm}

The next theorem gives some lower and upper bounds for the total dominator chromatic number of the central of a disconnected graph in which it is supposed that none of its connected components is $K_1$. See Theorem \ref{join graph} when one connected component of the graph is a single vertex. 

\begin{thm}
\label{n-w+1}
Let $G$ be a graph of order $n\geq 2$ with $\delta(G)\geq 1$. If $G=G^1 \cup G^2\cup \cdots \cup G^w$, that is $ G^{1}$, $G^{2}$, $\cdots$, $G^{w}$ are all connected components of $ G $, for some $w\geq 2$, then $\chi_d ^t(C(G))$ has the following tight bounds:
\[
\sum\limits_{i=1}^{w}\lfloor (2|G^{i}|/3 \rfloor +1 \leq \chi_d ^t(C(G)) \leq n+w-1.
\]
\end{thm}

\begin{proof}
Let $G=G^1 \cup G^2\cup \cdots \cup G^w$ be a graph of order $n\geq 2$ with $\delta(G)\geq 1$ in which $ G^{1}$, $G^{2}$, $\cdots$, $G^{w} $ are all connected components of $G$ and $|G^i|=n_i\geq 2$ for $1\leq i \leq w$. Obviously $C(G)$ is a graph which is obtained by replacing every maximal independent set of cardinality $n_i$ in $K_{n_1, n_2,\cdots,n_m}$ by $C(G^i)$. If $ V(G^{i})=\{v^{i}_{1}, v^{i}_{2},\cdots, v^{i}_{n_i}\}$ and $ \mathcal{C}_i=\{c^i_{i'j'}~|~ v^i_{i'}v^i_{j'}\in E(G^i) \}$ for $1\leq i \leq w$, then 
\[
V(C(G))=V(G^1)\cup V(G^2)\cup \cdots \cup V(G^w) \cup \mathcal{C}_1\cup \cdots \cup \mathcal{C}_w.
\] 
Let $ f $ be a coloring function on $ V(C(G)) $ such that for any $ 1\leq i\leq w $, 
\[
f(v^{i}_{n_{i} -1})=f(v^{i}_{n_i})=\sum\limits_{j=1}^{i}n_{j}-i , 
\]
\begin{equation*}
f(v^{i}_{k})=\left\{
\begin{array}{ll}
k         & \mbox{ if } i=1,\\
\sum\limits_{j=1}^{i-1}n_{j}-i+k+1       & \mbox{ if } i\geq 2,
\end{array}
\right.
(\mbox{where } 1\leq k\leq n_i -2), 
\end{equation*}
\[
f(c^{i}_{i^{'}  j^{'}})=\sum\limits_{j=1}^{w}n_{j}-w+1 ~\mbox{ for } c^{i}_{i^{'}  j^{'}} \in \mathcal{C}_i.
\] 
Since $ f $ is a TDC of $C(G)$, $\chi_d ^t(C(G))\leq \sum\limits_{i=1}^{w}n_{i}-w+1$. 

\vspace{0.15cm}
As we saw in the proof of Theorem \ref{n/2+2 =< chi_d t(C(G))=<n/2+n}, at least $\lfloor 2n_i/3 \rfloor$ new colors are needed to color the vertices of $G^i$. Since also a new color is needed to color the vertices of $\mathcal{C}$, we obtain $\chi_d ^t(C(G)) \geq \sum\limits_{i=1}^{w}\lfloor 2n_i/3\rfloor +1 $.

\vspace{0.2cm}

This upper bound is sharp for $ G=K_{n_1} \cup K_{n_2}\cup \cdots \cup K_{n_w}$. Because it can be easily proved that in every min-TDC of $C(G)$ the number of needed colors to color the vertices of each $K_{n_i}$ are $n_i -1$ that do not appear in the other components. So $ \chi_d^t(C(G)) \geq n-w$. On the other hand, since at least one color is needed to color the vertices in $\mathcal{C}_1\cup \cdots \cup \mathcal{C}_w$, we obtain $\chi_d^t(C(G))=n-w+1$. Also, the lower bound is sharp for $G=P_{n_1} \cup P_{n_2}\cup \cdots \cup P_{n_w}$ when each $n_i\not\equiv 1 \pmod 3$ and $n_i \geq 6$, and also for $G=C_{n_1} \cup C_{n_2}\cup \cdots \cup C_{n_w}$ when each $n_i \equiv 0 \pmod 3$ and $n_i \geq 6$.
\end{proof}

The reader can easily prove that for any nontrivial connected graph $G$, $ \chi_d ^t(C(G))=2$ if and only if $G\cong K_2$, and 
for any nontrivial connected graph $G\not\cong K_2$, $ \chi_d ^t(C(G))\geq 4$. So 

\begin{rem}
There is no connected graph of order $n\geq 2$ with $ \chi_d ^t(C(G))=3$.
\end{rem}

\begin{prop}
\label{chi_d t(C(G))=n}
For any $n\neq 1,3$, there exists a connected graph $ G $ of order $n$ with $\chi_d ^t(C(G))=n $.
\end{prop}

\begin{proof}
Since obviousely $\chi_d ^t(C(K_2))=2$, we assume $n\geq 4$. We show that $\chi_d ^t(C(G))=n$ where 
\[
G=K_n-(\mbox{ a path } P_2 \mbox{ and a maximum matching) for even } n,
\]
\[
G=K_n-(\mbox{ a path } P_3 \mbox{ and a maximum matching) for odd } n.
\]
Without loss of generality, let
\[
G=K_{n}-(\{v_{1}v_{4}\}\cup \{v_{2i-1}v_{2i}~|~ 1\leq i \leq \lfloor n/2 \rfloor\}) \mbox{ for even } n,
\]
\[
G=K_{n}-(\{v_{1}v_{4},v_{4}v_{n}\} \cup \{v_{2i-1}v_{2i}~|~ 1\leq i \leq \lfloor n/2 \rfloor\}) \mbox{ for odd } n.
\]
Then for any TDC $f=(V_{1}, V_{2},\cdots,V_{\ell} ) $ of $ C(G)$, the number of color classes $V_i\subset V(G)$ of cardinality one is at least $n-2$. Because in the otherwise, there exist two color classes $V_1=\{v_{i}, v_{j}\} $ and $V_2=\{v_{k}, v_{t}\}$ such that $ N(v_{i})\cap  N(v_{k})=\{c_{ik}\}$ and $c_{ik}\nsucc_{t} V_{m} $ for each $ 1\leq m \leq \ell $, a contradiction. Therefore $ |\{f(v_{i})~|~ 1\leq i \leq n\}|\geq n-1 $, and since at least one new color is needed to color some vertices in $\mathcal{C}=\{c_{ij}~|~v_iv_j\in E(G)\}$, we have $\ell \geq n$. Now since 
\[
f=(\{v_{1}\},\{v_{2}, v_{3} \},\{v_4\},\cdots, \{v_n\},\mathcal{C}) 
\]
is a TDC of $ C(G) $, we have $ \chi_{d}^{t}(C(G)=n$. In Figure \ref{fi:proofexample8}, $(\{v_1\},\{v_2,v_3\},\{v_4,v_5\},\{v_6\},\mathcal{C})$ is a min-TDC of the left graph, and $(\{v_1\},\{v_2,v_3\},\{v_4\},\{v_5\},\{v_6\},\{v_7\},\mathcal{C})$ is a min-TDC of the right graph.
\end{proof}

\begin{figure}[ht]
\centerline{\includegraphics[width=8cm, height=4cm]{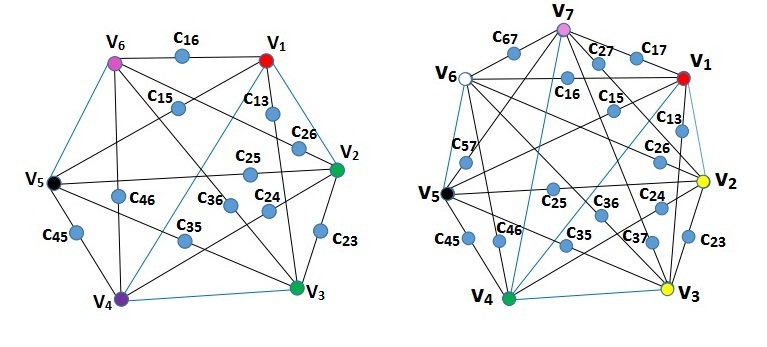}}
\vspace*{-0.5cm}
\caption{A min-TDC of the central of two graphs of orders $n=6,7$ with $n$ colors}\label{fi:proofexample8}
\end{figure}
\section{The join of two graphs}

Here, we will find bounds for the total dominator chromatic number of the central of join of a graph with an empty graph $K_t$. We recal that the \emph{join} $G\circ H$ of two graphs $G$ and $H$ is the graph obtained by the disjoint union of $G$ and $H$ joining each vertex of $G$ to all vertices of $H$.

\begin{thm}
\label{join graph}
For any graph $G$ of order $ n\geq2 $,
\[
\chi^{t}_{d}(C(G) )+t \leq \chi^{t}_{d}(C(G\circ \overline{K_{t}}) )\leq \chi^{t}_{d}(C(G) )+t+1.
\]
\end{thm}

\begin{proof} For any integers $ n\geq 2 $ and $t\geq 1$, let $G$ be a graph with $V(G)=\{v_{1},\cdots,v_{n}\} $ and let $V(\overline{K_{t}})=\{v_{n+1},\cdots,v_{n+t}\} $. Then $V(C(G\circ \overline{K_{t}}))=V(G\circ \overline{K_{t}})\cup \mathcal{C}_{1}\cup \mathcal{C}_{2}$ where $\mathcal{C}_{1}=\{ c_{i  j}~|~ v_i v_j \in E(G) \}$ and $\mathcal{C}_{2}=\{ c_{(n+i)j}~|~1\leq i \leq t, ~~1\leq j \leq n \}$. Since for any min-TDC $ f=(V_{1}, V_{2},\cdots,V_{\ell} )$ of $C(G) $, the coloring function $g=(V_{1}, V_{2},\cdots,V_{\ell}, \mathcal{C}_2, \{v_{n+1}\},\cdots,\{v_{n+t}\})$ is a TDC of 
$C(G\circ \overline{K_{t}})$, we have $\chi^{t}_{d}(C(G\circ \overline{K_{t}}) )\leq \chi^{t}_{d}(C(G) )+t+1$, as desired.

\vspace{0.2cm}

Now we prove the lower bound in the following two cases.
\vspace{0.2cm}

\textbf{Case 1. } $t=1$. Let $ f=(V_{1}, V_{2},\cdots,V_{\ell} ) $ be  a min-TDC of $ C(G\circ K_{1}) $ such that $ v_{n+1}\in V_{\ell } $ and $v_{n+1}\succ_{t} V_{\ell -1}$ (and so $ V_{\ell-1}\subseteq \mathcal{C}_2$). First we state an algorithm.

\vspace{0.2cm}

 \textbf{Changing Color Algorithm (CCA):}\\
- step 1. Choose a vertex $v_i$ with this property that $ v_{i}\succ_{t} V_{j} $ implies $V_{j}=\{ c_{(n+1) i}\}$.\\
- step 2. If $v_i$ is a vertex choosen in step 1, change the color of one vertex $w_i$ in $N_{C(G)}(v_i)$ to the color $f(c_{(n+1)i})$.

\vspace{0.2cm}

If $ |V_{\ell }|=1$, then by using CCA we find a TDC of $C(G)$ with $\ell-1$ colors $ 1,2,\cdots,\ell -1 $, and so $ \chi^{t}_{d}(C(G) )\leq \ell -1  $, as desired. So, we assume $ |V_{\ell }|\geq 2$. Let 
\[
T=\{v_{i}\in V(G)~|~v_{i}\succ_{t} V_{j} ~\mbox{implies}~V_{j}=\{c_{(n+1) i}\}\}.
\]
Since the restriction of $f$ on $V(C(G))$, that is,
\[
f_{|_{V(C(G))}}=(V_{1}\cap V(C(G)),\cdots,V_{\ell-2}\cap V(C(G)),V_{\ell}\cap V(C(G))),
\]
is a TDC of $ C(G) $ with $ \ell -1 $ color classes when $ T=\emptyset $, and so $ \chi^{t}_{d}(C(G) )\leq \ell -1  $, as desired, we assume $ T\neq \emptyset $. By using CCA and restriction of $f$ on $V(C(G))$, we obtain a TDC $f_{0}$ of $C(G)$ with $\ell$ color classes. In the following subcases, by improving $ f_{0}$, we will find a TDC of $ C(G) $ with at most $ \ell-1 $ color classes, and this completes our proof.

 \vspace{0.15cm}
 
 \textbf{Subcase 1.1.} Let $ V_{\ell }\cap\{v_{1},\cdots,v_{n}\}\neq \emptyset $. Without loss of generality, we may assume $ v_{1} \in T $, and $ v_{k} \in V_{\ell } $. Then $ V_{\ell }=\{v_{n+1},v_{k}\} $ because of $ N_{C(G\circ K_{1})}(c_{(n+1)k})=\{v_{n+1},v_{k}\} $ and $ c_{(n+1)k} \succ_{t} V_{\ell} $. Let $ v_{i} \in T-\{v_{k}\} $ for some $i$. If $v_{i}v_{k}\in E(G)$, then $ |V_{t}|\geq2 $ (by assumption $c_{ik}\in V_t$ and the definition of $ T $). Hence the coloring function $ g $ on $ C(G) $ with the criterion
 \begin{equation*}
g(x)=\left\{
\begin{array}{ll}
f_{0}(x)                                          & \mbox{if }x \not\in V_{t}, \\
f(c_{(n+1)i})  & \mbox{if }x \in V_{t}.
\end{array}
\right.
\end{equation*}
is a TDC of $ C(G) $ with $ \ell-1 $ color classes, as desired. Also if $ v_{i}v_{k}\not\in E(G) $, then $ v_{i}v_{k}\in E(C(G)) $, and the coloring function $ h $ on $ C(G) $ with the criterion
\begin{equation*}
h(x)=\left\{
\begin{array}{ll}
f_{0}(x)   & \mbox{if }x \neq v_{k},\\
f(c_{(n+1)i})  & \mbox{if }x = v_{k}, 
\end{array}
\right.
\end{equation*}
is a TDC of $ C(G) $ with $ \ell-1 $ color classes, as desired. Therefore, we consider $ T=\{v_{k}\}=\{v_{1}\} $. Then the coloring function $ p $ on $ C(G) $ with the criterion
\begin{equation*}
p(x)=\left\{
\begin{array}{ll}
f_{0}(x)   & \mbox{if }x \neq v_{k},\\
\ell-1       & \mbox{if }x = v_{k}, 
\end{array}
\right.
\end{equation*}
is a TDC of $ C(G) $ with $ \ell-1 $ color classes, as desired.
\vspace{0.15cm}

\textbf{Subcase 1.2.} $ V_{\ell }\cap\{v_{1},\cdots,v_{n}\}=\emptyset $. Then $ Q=\{c_{ij}~|~ v_{i}v_{j}\in E(G)\}\cap V_{\ell }\neq \emptyset $, and so the function
\begin{equation*}
q(x)=\left\{
\begin{array}{ll}
f_{0}(x)   & \mbox{if }x \not\in Q,\\
\ell-1       & \mbox{if }x \in Q, 
\end{array}
\right.
\end{equation*}
 is a TDC of $ C(G) $ with $ \ell-1 $ color classes, as desired.
\vspace{0.2cm}

\textbf{Case 2.} $t\geq2 $. Let $f=(V_{1},\cdots,V_{\ell})$ be a min-TDC of $C(G\circ \overline{K_{t}})$. Let $v_i \succ V_j$ for some $1 \leq i \leq n$ and $1\leq j \leq \ell$ such that $V_j\cap \mathcal{C}_2\neq \emptyset$. We see that if $|V_j|\geq 2$, then $v_i \succ V_j'$ where $V_j'=V_j-\mathcal{C}_2$, and if $V_j=\{c_{(n+i)m}\}$ for some $c_{(n+i)m}\in \mathcal{C}_2$, then there exists a vertex $c_{iq}\in \mathcal{C}_1$ such that by changing its color to the color $f(c_{(n+i)m})$ we have $v_i \succ V_j'$ where $V_j'=\{c_{iq}\}$. This shows that $f_{|_{V(C(G))}}$, the restriction of $f$ on $V(C(G))$, is a TDC of $C(G)$. On the other hand, we know that
\[
|\{f(n+i)~|~ 1 \leq i \leq t\}|=t  \mbox{ and}
\]
\[
\{f(x)~|~x\in V(C(G))\} \cap \{f(n+i)~|~ 1 \leq i \leq t\}=\emptyset.
\]
Therefore $\chi^{t}_{d}(C(G\circ \overline{K_{t}}) ) \geq \chi^{t}_{d}(C(G) )+t$.
\end{proof}


\section{Nordhaus-Gaddum-like relations}

Finding a Nordhaus-Gaddum-like relation for any parameter in graph theory is one of a tradition work which is started after the following theorem by Nordhaus and Gaddum in 1956 \cite{Nordhaus}.

\begin{thm}
\label{Nordhaus and Gaddum}
\emph{\cite{Nordhaus}} For any graph $G$ of order $n$, $2\sqrt{n}\leq \chi(G)+\chi(\overline{G}) \leq n+1$.
\end{thm}

Here, we will find Nordhaus-Gaddum-like relations for the total dominator chromatic number.
\begin{thm}
\label{chi^{t}_{d}(overline{C(G)})}
For any connected graph $G$ of order $n\geq4$ and size $ m $,
\begin{equation*}
\chi^{t}_{d}(\overline{C(G)})=\left\{
\begin{array}{ll}
n                              & \mbox{if } G \mbox{ is a tree}, \\
m  ~~  & \mbox{otherwise}.
\end{array}
\right.
\end{equation*}
\end{thm}
  
\begin{proof}
Let $ G $ be a connected graph of order $ n\geq 4 $ and size $ m\geq 3$ with the vertex  $ V=\{v_{1}, v_{2},\cdots, v_{n}\}$. Then $V(C(G))=V(\overline {C(G)})=V\cup \mathcal{C}$ where $\mathcal{C}=\{ c_{i  j}~|~v_{i}v_{j}\in E(G) \}$ and $ E(\overline {C(G)})=E(G)\cup\{c_{ij}v_{k}~|~c_{ij}\in \mathcal{C},~v_k\in V, \mbox{ and }  k\neq i ,j\}$. Since the subgraph of $\overline{C(G)}$ induced by $\mathcal{C}$ is a complete graph of order $m$, we have $\chi (\overline{C(G)})\geq m$. Now we continue our proof in the following two cases. 

\vspace{0.2cm}

\textbf{Case 1.} $m=n-1$. Let $ f=(V_{1},\cdots,V_{m})$ be a proper coloring of $ \overline{C(G)}$. Then by the piegonhole principle $v_{i}, v_{j}\in V_{k}$ for some $i\neq j$ and some $1\leq k \leq m$, and so $v_{i}v_{j}\not\in E(G)$. Since the subgraph of $\overline{C(G)}$ induced by $\mathcal{C}$ is isomorphic to the complete graph $K_{m}$, we have $f(c_{pq})=k$ for some $c_{pq}\in \mathcal{C}$ in which $(p,q)\neq(i,j)$. Then $\{c_{pq}v_{i},c_{pq}v_{j}\} \cap E(\overline{C(G)})\neq \emptyset$ (because $c_{pq}v_{i}\in E(C(G))$ implies $p=i$ and $q\neq j$, and so $c_{pq}v_{j}\in E(\overline{C(G)})$ which contradicts the fact $f(v_{j})=f(c_{pq})=k$). Therefore $\chi (\overline{C(G)})\geq m+1$ and so $\chi_d^t (\overline{C(G)})\geq m+1$. Now by assumptions $v_{1}v_{n}\in E(\overline{C(G)})$ and  $v_{1}v_{2}\not \in E(\overline{C(G)})$, since the coloring function $g=(V_{1},\cdots,V_{m}, V_{m+1})$ is a TDC of $\overline{C(G)}$ in which $V_{1}=\{v_{1},v_{2}\}$, $V_{m+1}=\{c_{1n}\}$, and $V_{i}=\{v_{i+1}\}\cup \{c_{pq}~|~p \mbox{ or } q \mbox{ is }i+1 \mbox{ and } p+q \mbox{ is minimum and }c_{pq}\not\in V_1\cup ~\cdots ~\cup V_{i-1}\}$ for $2\leq i \leq m$, we have $\chi^{t}_{d} (\overline{C(G)})= m+1$. 
\vspace{0.2cm}

For an example see Figure \ref{fi:proofexample9} \textcolor{red}{(a)} in which $(\{v_1,v_2\},\{v_3,c_{35}\},\{v_4c_{24}\},\{v_5,c_{25}\},\{c_{15}\})$ is a min-TDC of the graph.

\vspace{0.2cm}

\textbf{Case 2.} $m\geq n$. For $m= n$, consider the coloring function $ f=(V_{1},\cdots,V_{m})$ in which $V_i=\{v_i\}\cup \{c_{pq}~|~p \mbox{ or } q \mbox{ is } i \mbox{ and } p+q \mbox{ is minimum and }c_{pq} \not\in V_1\cup \cdots \cup V_{i-1}\}$ for $1\leq i \leq m$ and for $m> n$ consider the coloring function $ f=(V_{1},\cdots,V_{m})$ in which
\[
V_{i}=\{v_{i}\}\cup \{c_{pq}~|~p \mbox{ or } q \mbox{ is } i \mbox{ such that } p+q \mbox{ is minimum and }c_{pq}\not\in V_1\cup ~\cdots ~\cup V_{i-1}\}~(1\leq i \leq n),
\]
\[
V_{n+i}=\{\alpha_i\} \mbox{ when } 1\leq i\leq m-n \mbox{ and } \mathcal{C}-(V_{1}\cup \cdots \cup V_{n})=\{\alpha_i~|~ 1\leq i\leq m-n \}.
\]
Since in each of the cases the coloring functions $f$ are total dominating colorings of $\overline{C(G)}$, we have $\chi^{t}_{d} (\overline{C(G)})= m$. 

\vspace{0.2cm}

In Figure \ref{fi:proofexample9}, $(\{v_1,c_{12}\},\{v_2,c_{23}\},\{v_3,c_{34}\},\{v_4,c_{14}\})$ is a min-TDC of the middle graph, and $(\{v_1,c_{12}\},\{v_2,c_{23}\},\{v_3,c_{34}\},\{v_4,c_{24}\},\{v_5,c_{35}\})$ is a min-TDC of the right graph.
\end{proof}
\begin{figure}[ht]
\centerline{\includegraphics[width=12cm, height=5cm]{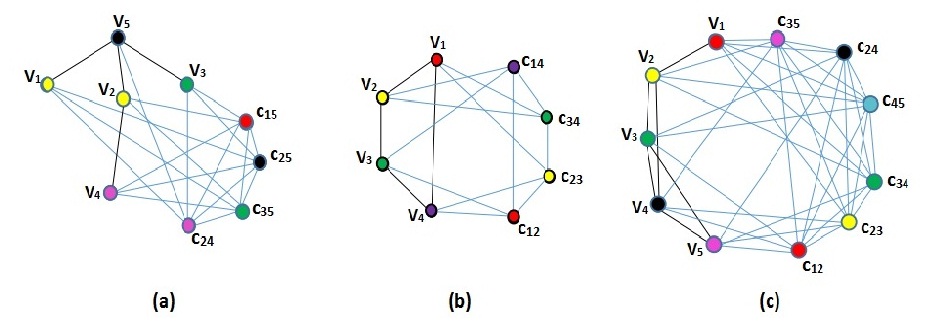}}
\vspace*{-0.3cm}
\caption{A min-TDC of $\overline{C(G)}$ when $m=n-1$ (left), $m=n$ (middle) and $m>n$ (right)}\label{fi:proofexample9}
\end{figure}


As a result of Theorems \ref{2n/3+1 =< chi_dt C(G) =< n+n/2}, \ref{2n/3 +1 =< chi_d ^t(C(G)) =<n+1}, \ref{chi^{t}_{d}(overline{C(G)})} and Proposition \ref{chi_d^t C(P_n)}  we have the next propositions as two Nordhaus-Gaddum relations.

\begin{prop}
For any tree $\mathbb{T}$ of order $n\geq 4$,
\begin{equation*}
\chi_d ^t(C(\mathbb{T})+ \chi_d ^t(\overline{C(\mathbb{T})})=\left\{
\begin{array}{ll}
\lfloor 2n/3\rfloor+n+2   & \mbox{if }n\equiv 1 \pmod{3}\mbox{ or }n=5, \\
\lfloor 2n/3\rfloor+n+1   & \mbox{otherwise},
\end{array}
\right.
\end{equation*}
if $\mathbb{T}$ is a path, and 
\[
n+ 1+\lfloor 2n/3 \rfloor \leq \chi_d ^t(C(\mathbb{T}))+ \chi_d ^t(\overline{C(\mathbb{T})})\leq 2n+1
\]
if $\Delta(\mathbb{T}) \leq n-2$,
\end{prop}

\begin{prop}
For any connected graph $G$ of order $n\geq4$ and size $ m\geq n$,
\[
m+1+\lfloor 2n/3 \rfloor \leq \chi_d ^t(C(G))+ \chi_d ^t(\overline{C(G)})\leq m+n+\lceil n/2 \rceil
\]
if $G$ has a Hamiltonian path, and 
\[
m+1+\lfloor 2n/3 \rfloor \leq \chi_d ^t(C(G))+ \chi_d ^t(\overline{C(G)})\leq m+n+1
\]
if $\Delta(G) \leq n-2$.
\end{prop}

\section{\bf Cycles, paths, wheels and complete multipartite graphs}\label{t80}
In this section, we calculate the total dominator chromatic number of the central of cycles, paths, wheels and multipartite graphs. The total dominator chromatic number of the central of cycles and paths are given in the first two proposions. 
\begin{prop}
\label{chi_d^t C(P_n)}
For any path $P_n$ of order $n\geq 2$,
\begin{equation*}
\chi_d^t(C(P_n))=\left\{
\begin{array}{ll}
\lfloor 2n/3\rfloor+2   & \mbox{if }n\equiv 1 \pmod{3}\mbox{ or }n=3, 5, \\
\lfloor 2n/3\rfloor+1   & \mbox{otherwise}.
\end{array}
\right.
\end{equation*}
\end{prop}

\begin{proof}
Since $ C(P_{2})\cong P_{3} $ and $ C(P_{3})\cong C_{5} $, and obviousely $ \chi^{t}_{d}(C({P_2}))= 2,~ \chi^{t}_{d}(C({P_3}))= 4 $, we assume $n\geq 4$. Let $ P_{n}:v_{1}v_{2}\cdots v_{n}$ be a path of order $ n\geq2 $ in which $v_{i}v_{j}\in  E(P_{n}) $ if and only if $2\leq j =i+1 \leq n$. Then $V(C(P_{n}))=V\cup \mathcal{C}$ where $V=V(P_n)$ and $\mathcal{C}=\{ c_{i(i+1)}~|~1\leq i \leq n-1 \}$.  Then by Theorem \ref{n/2+2 =< chi_d t(C(G))=<n/2+n} and this fact that the coloring function $ f $ with the criterion
\begin{equation*}
  f(v_{i})=\left\{
\begin{array}{ll}
2\lfloor i/3\rfloor                                & i\equiv 0 \pmod 3, \\
n-\lfloor n/3\rfloor    & i=n,\\
2\lfloor i/3\rfloor +1 & otherwise,
\end{array}
\right.
\end{equation*}
\[
f(c_{i(i+1)})=n-\lfloor\dfrac{n}{3}\rfloor +1 \mbox{ (for } 1\leq i \leq n)
\]
is a TDC of $C(P_n)$, we obtain
\begin{eqnarray}
\label{2n/3 +1 =< chi t_d(C(P_n)) =< n-n/3+1}
\lfloor\dfrac{2n}{3}\rfloor +1 \leq \chi^{t}_{d}(C(P_{n})) \leq n-\lfloor\dfrac{n}{3}\rfloor+1.
\end{eqnarray}

\vspace{0.2cm}

\textbf{Case 1.} $n\equiv 2\pmod{3}$. First we prove $ \chi^{t}_{d}(C(P_{5}))=5$. Let $ f=(V_{1}, V_{2},\cdots, V_{\ell}) $ be a min-TDC of $C(P_5)$ for some $\ell$ (notice: $4\leq \ell \leq 5$ by Proposition \ref{2=<chi_d ^t=<n}). Without loss of generality, we may assume $V_1=\{v_1,v_2\}$, $V_2=\{v_3\}$, and $V_3=\{v_4,v_5\}$. Let $v_3 \succ_{t} V_{j}$ for some $j$. Then $j\geq 4$ (say $j=4$), and so $V_{4}\subseteq \{c_{23},c_{34}\}$. Since $\{c_{12},c_{45}\}\cap (V_1\cup \cdots \cup V_4)=\emptyset$, we have $\ell=5$. Now for $ n\neq 5 $, since the coloring function $ f $ of $C(P_n)$ with the criterion
\begin{equation*}
  f(v_{i}) =\left\{
\begin{array}{lll}
2\lfloor i/3 \rfloor                                & i\equiv 0 \pmod 3,& \\
2\lfloor i/3 \rfloor +1 & i\not\equiv 0 \pmod 3, & \mbox{and}
\end{array}
\right.
\end{equation*}
\[
 f(c_{i(i+1)})= \lfloor 2n/3\rfloor+1 \mbox{ (for } 1\leq i \leq n-1)
\] 
is a TDC of $ C(P_{n}) $ with $\lfloor 2n/3\rfloor+1 $ color classes, we obtain $ \chi_d^t(C(P_n))= \lfloor 2n/3\rfloor+1 $ by (\ref{2n/3 +1 =< chi t_d(C(P_n)) =< n-n/3+1}).
  
\vspace{0.2cm}

\textbf{Case 2.} $n \not\equiv 2 \pmod{3}$. Since $\lfloor 2n/3 \rfloor +1= n-\lfloor n/3 \rfloor+1$ in (\ref{2n/3 +1 =< chi t_d(C(P_n)) =< n-n/3+1}) when $n\equiv 0 \pmod{3}$, it is sufficient to prove $ \chi_d^t(C(P_n)) > n-\lfloor \frac{n}{3}\rfloor $ when $n\equiv 1 \pmod{3}$. Let $ f=(V_{1}, V_{2},\cdots, V_{\ell} ) $ be a min-TDC of $ C(P_{n})$, and let 
\[
J=\{j\mid 1\leq j \leq \ell,~c_{i(i+1)}\succ_{t} V_{j} \mbox{ for some } 1\leq i \leq n-1\}.
\]
Then $ J_{1}\cup J_{2} $ is a partition of $ J $ where $ J_{i}=\{ j \in J \mid |V_{j}|=i \} $ for $i=1,2$. If $ J_{2}=\emptyset $, then $ \ell \geq |J| \geq n $, and there is nothing to prove. Hence $ J_{2}\neq \emptyset $, and $ n= 2|J_{2}|+|J_{1}| $ implies $ \ell \geq |J| = |J_{2}|+|J_{1}|= n- |J_{2}|$. Let $ |J| \leq n-\lfloor \frac{n}{3}\rfloor-1 $. Then $ |J_{2}|\geq \lfloor \frac{n}{3}\rfloor+1 $. Since $ V_{t}=\{ v_{i}, v_{i+1} \} $ (for some $ t $) implies $ V_{k}=\{ v_{i+2} \} $ (for some $ k $), we conclude $ |J_{1}|\geq |J_{2}|  $, and so 
\[
n= 2|J_{2}|+|J_{1}|\geq 3|J_{2}| \geq 3\lfloor \frac{n}{3}\rfloor+3>n,
\]
a contradiction. Therefore $  |J| \geq n-\lfloor \frac{n}{3}\rfloor $. On the other hand, since there exists at least a color class $V_{t} $ such that $V_{t} \cap V(P_{n})= \emptyset $, we obtain $ \ell > n-\lfloor \frac{n}{3}\rfloor$, as desired. 

\vspace{0.2cm}

In Figure \ref{fi:proofexample7}, $(\{v_1,v_2\},\{v_3\},\{v_4,v_5\},\{v_6\},\{v_7,v_8\},\mathcal{C})$ is a min-TDC of $C(P_8)$.
\end{proof}

\begin{figure}[ht]
\centerline{\includegraphics[width=4.5cm, height=4cm]{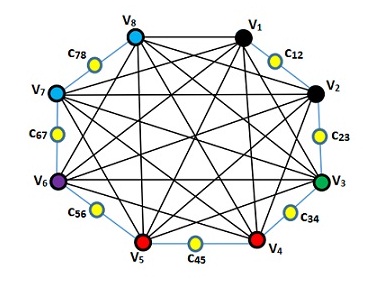}}
\vspace*{-0.5cm}
\caption{A min-TDC of $C({P_8})$}\label{fi:proofexample7}
\end{figure}

\begin{prop}
\label{chi_d^t(C({C_n}))}
For any cycle $C_n$ of order $n\geq 3$,
\begin{equation*}
\chi_d^t(C({C_n}))=\left\{
\begin{array}{ll}
\lfloor 2n/3\rfloor +1   & \mbox{if }n\equiv 0 \pmod 3  \mbox{ and }n\neq 3,\\
\lfloor 2n/3\rfloor +2  & \mbox{otherwise}.
\end{array}
\right.
\end{equation*}
\end{prop}

\begin{proof}
Let $ C_{n}:v_{1}v_{2}\cdots v_{n}$ be a cycle of order $ n\geq3 $ in which $v_{i}v_{j}\in  E(C_{n}) $ if and only if $ j \equiv i+1 \pmod n$. Then $V(C(C_{n}))=V\cup \mathcal{C}$ where $\mathcal{C}=\{ c_{i(i+1)}~|~1\leq i \leq n \}$. We know $\chi_d^t(C({C_n})) \geq \lfloor 2n/3\rfloor +1$ by Theorem \ref{n/2+2 =< chi_d t(C(G))=<n/2+n}. Obviousely $C(C_{3})\cong\ C_{6}$ and so $\chi^{t}_{d}({C(C_3}))= 4$. For $n=4$, it can be easily verified $ \gamma_{t}(C(C_{4}))=4$, and so $\chi^{t}_{d}(C(C_{4}))\geq \gamma_{t}(C(C_{4}))=4$ by Proposition \ref{2=<chi_d ^t=<n}. Now since $ f=(\{v_{1}, c_{23}, c_{34} \},\{v_{2}\},\{v_{3}, c_{12}, c_{41} \},\{v_{4}\}) $ is a TDC of $ C(C_{4}) $, we obtain $ \chi^{t}_{d}(C(C_{4}))= 4$. Now let $ n\geq 5$. Since $ f $ is a TDC of $ C(C_{n}) $ with the criterion
\begin{equation*}
  f(v_{i})=\left\{
\begin{array}{ll}
n-\lfloor n/3\rfloor    & i=n,\\
2\lfloor i/3\rfloor                                & i\equiv 0 \pmod 3, \\
2\lfloor i/3\rfloor +1 & otherwise,
\end{array}
\right.
\end{equation*}
\begin{equation*}
  f(c_{i(i+1)})=\left\{
\begin{array}{ll}
\lfloor 2n/3\rfloor +1                                & n\equiv 0 \pmod 3, \\
\lfloor 2n/3\rfloor +2 & otherwise,
\end{array}
\right.
\end{equation*}
for  $1\leq i \leq n$, we have
\begin{equation*}
\chi_d^t(C({C_n}))\leq \left\{
\begin{array}{ll}
\lfloor 2n/3\rfloor +1   & \mbox{if }n\equiv 0 \pmod 3  \mbox{ and }n\neq 3,\\
\lfloor 2n/3\rfloor +2  & \mbox{otherwise},
\end{array}
\right.
\end{equation*}
and so there is nothig to prove when $ n\equiv 0 \pmod 3 $. Therefore, we assume $ n\not\equiv 0 \pmod 3 $. Let $ f=(V_{1}, V_{2},\cdots, V_{\ell} ) $ be a min-TDC of $ C(C_{n})$ and let $v_1\in V_1$. Since there exists an unique index $i$ such that $V_{i}=\{f(v_{n})\} $ where $ n\equiv 1 \pmod 3 $, and there exist two different indices $i$ and $j$ such that $V_{i}=\{f(v_{n-1})\}$ and $V_{j}=\{f(v_{n})\}$ where $ n\equiv 2 \pmod 3 $, we obtain $ \chi^{t}_{d}(C(C_{n}))\geq \lfloor 2n/3\rfloor +2$. Now our proof is completed. 

\vspace{0.2cm}

In Figure \ref{fi:proofexample6}, the coloring function $(\{v_1,v_2\},\{v_3\},\{v_4,v_5\},\{v_6\},\{v_7\},\{v_8\},\mathcal{C})$ is a min-TDC of $C(C_8)$.

\begin{figure}[ht]
\centerline{\includegraphics[width=4.5cm, height=4cm]{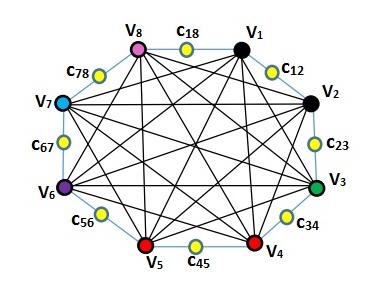}}
\vspace*{-0.3cm}
\caption{A min-TDC of $C({C_8})$}\label{fi:proofexample6}
\end{figure}

\end{proof}

The total dominator chromatic number of the central of a wheel is considered in the next proposition.
\begin{prop}
For any wheel $W_n$ of order $n+1\geq 4$,
 \begin{equation*}
\chi_d^t(C({W_n}))=\left\{
\begin{array}{ll}
\lfloor 2n/3\rfloor +3   & \mbox{if }n\equiv 0 \pmod 3  \mbox{ and }n\neq 3,\\
\lfloor 2n/3\rfloor +4  & \mbox{otherwise}.
\end{array}
\right.
\end{equation*}
\end{prop}

\begin{proof}
Since $W_3$ is isomorphic to the complete graph $K_4$, and $\chi_d^t(C(K_4))=6$ by Theorem \ref{chi_d ^t(C(G))=n+ lceil n/2 rceil}, we consider $W_n$ be a wheel graph of order $n+1\geq 5$ with the vertex set $ V=\{v_{i}\mid 0 \leq i \leq n \} $, and the edge set $ E=\{v_{0}v_{i}, v_{i}v_{i+1}~|~ 1\leq i \leq n \}$. Then $ V(C(W_{n}))=V\cup \mathcal{C}$ where $\mathcal{C}=\{ c_{0i}, c_{i(i+1)}~|~1\leq i \leq n\}$. Since $ W_{n}=C_{n}\circ K_{1} $ where $ V(K_{1})=\{v_{0}\} $ and $V(C_{n})=V-\{v_0\}$,  Theorem \ref{join graph} implies 
\[
\chi_d^t(C({C_n}))+1 \leq \chi_d^t(C({W_n})) \leq \chi_d^t(C({C_n}))+2,
\]
and so it is sufficient to prove $ \chi_d^t(C({W_n}))= \chi_d^t(C({C_n}))+2 $ by Proposition \ref{chi_d^t(C({C_n}))}. Let $ f=(V_{1}, V_{2},\cdots,V_{\ell} ) $ be a min-TDC of $ C(W_{n}) $ where $ \ell= \chi_d^t(C({C_n}))+1$. Without loss of generality, we may assume $v_0 \succ_{t} V_{\ell}$ and $v_{0}\in V_{\ell-1}$, which imply $V_{\ell} \subseteq \{ c_{0i}~|~1\leq i \leq n\}$. Since $ V_{\ell-1}=\{v_{0}\} $ implies
\begin{eqnarray}
\label{{f(c_i(i+1)),f(v_i)|1=< i =< n } subset V_1 ... V_l-2}
\{f(c_{i(i+1)}),f(v_{i})~|~ 1\leq i \leq n \}\subseteq V_1\cup \cdots \cup V_{\ell -2},
\end{eqnarray}
that contradicts the facts
\[
|\{f(v_i)~|~ 1\leq i \leq n \}|=n-\lfloor n/3 \rfloor=\ell-2, \mbox{ and}
\]
\[
\{f(c_{i(i+1)})\mid 1\leq i \leq n \} \cap \{f(v_{i})\mid 1\leq i \leq n \}= \emptyset,
\]
we assume $|V_{\ell-1}|\geq 2 $. If $v_{i}\in V_{\ell-1}$ for some $1\leq i \leq n$, then $ V_{\ell-1}=\{ v_{0}, v_{i} \} $, and so $ |V_{j}|=1 $ for each $ j\neq i $. Since $N_{C(C_n)}(v_{i})=\{ c_{i(i+1)},c_{(i-1)i}\}$, this implies the function $g$ with the criterion 
 \begin{equation*}
g(v)=\left\{
\begin{array}{ll}
f(v_{i-1})         & v=v_i,\\
f(v)        & \mbox{if }v\in V(C(C_n))-\{v_i\},
\end{array}
\right.
\end{equation*}
is a TDC of $C(C_n)$ with $\ell-2=\chi_d^t(C({C_n}))-1$ color classes, a contradiction. Therefore $ (V_{\ell-1}-\{v_0\})\subseteq  \{c_{i(i+1)}~|~ 1\leq i \leq n\}$, and so $\{\{v_i\}~|~1\leq i \leq n\}\subseteq \{V_i~|~1\leq i \leq \ell\}$. On the other hand, we have
\[
|\{f(c_{i(i+1)}),f(c_{0i})~|~ 1 \leq i \leq n\}\cup\{f(v_{0})\}|\geq2.
\]
Hence $\ell\geq n+2$, which is not possible for $n\geq 4$. Therefore $ \chi_d^t(C({W_n}))= \chi_d^t(C({C_n}))+2 $. 
\vspace{0.2cm}

In Figure \ref{fi:proofexample5}, $(\{v_0,v_{1}\},\{v_2,v_{3}\},\{v_4\},\{v_5\},\{c_{15},c_{12},c_{23},c_{34},c_{45}\},\{c_{01},c_{02},c_{03},c_{04},c_{05}\})$ is a min-TDC of $C({W_5})$.

\begin{figure}[ht]
\centerline{\includegraphics[width=5cm, height=4cm]{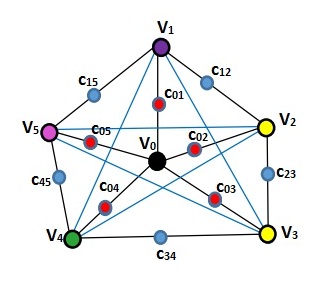}}
\vspace*{-0.3cm}
\caption{A min-TDC of $C({W_5})$}\label{fi:proofexample5}
\end{figure}

\end{proof}

Now, we consider the complete multipartite graphs. In the first step, we calculate the total dominator chromatic number of a complete bipartite graph.

\begin{prop}
\label{chi_d ^t (C(K_{m,n}))}
For any integers $n\geq m\geq 1$,
\begin{equation*}
\chi_d^t(C(K_{m,n}))=\left\{
\begin{array}{ll}
4 & \mbox{if } ~(m,n)= (1,2), \\
m+n     & \mbox{otherwise.}
\end{array}
\right.
\end{equation*}
\end{prop}

\begin{proof}
Let $K_{m,n}$ be a complete bipartite graph in which $n\geq m\geq 1$. Since the central graphs $C(K_{1,1})$ and $ C(K_{1,2})$ are isomorphic to $P_{3}$ and $C_{5}$, respectively, and obviousely $\chi^{t}_{d}(K_{m,n})=m+n~ \mbox{or}~ 4$ when $(m,n)$ is $(1,1)$ or $(1,2)$, respectively, we assume $(m,n)\not\in \{(1,1),(1,2)\}$. Consider $V\cup U$ as the partition of the vertex set of $K_{m,n}$ to the independent sets $V=\{v_{i}: 1\leq i \leq m\}$ and $U=\{u_{j}: 1\leq j \leq n\}$. Then $V\cup U\cup \mathcal{C}$ is a partition of the vertex set of $C(K_{m,n})$ in which $\mathcal{C}=\{c_{ij}~|~ 1\leq i \leq m, ~1\leq j \leq n  \}$. Let $ f=(V_{1}, V_{2},\cdots,V_{\ell}) $ be an arbitrary TDC of $C(K_{m,n})$.
Since the subgraph of $ C(K_{m,n})$ induced by $U$ is a complete graph of order $n$, we have $\ell \geq n$. In the following two cases, without loss of generality, we may assume $ u_{i}\in V_{i}$ for each $1\leq i \leq n$.

\vspace{0.15cm}
\textbf{Case 1.} $m=1  $.  Then $ n\geq 3$, by the assumption. If $ \ell = n $, then $ c_{1 \sigma(i)} \in V_{i} $ for each $ 1\leq i \leq n $ and some permutation $ \sigma $ on $ \{ 1, 2,\cdots,n \}$, which implies $ v_{1}\nsucc_{t} V_{i} $ for each $1\leq i \leq n$, a contradiction. Hence $\ell \geq n+1 $. Now since $ f=(V_{1}, V_{2},\cdots,V_{n+1}) $ is a TDC of $ C(K_{1,n}) $ where $ V_{i}=\{u_{i}\} $ for $ 1\leq i \leq n-1 $, $  V_{n}=\{v_{1}, u_{n}\}  $,  $V_{n+1}=\{c_{1i}~|~1\leq i \leq n \}$, we obtain $ \chi^{t}_{d}(C(K_{1,n}))=n+1$.

\vspace{0.15cm}

\textbf{Case 2.} $m\geq 2$. Then $ \vert V\cap (V_{1}\cup V_{2}\cup \cdots \cup V_{n}) \vert \leq 1, $ because if $ v_{i}\in V_{j} $ and $ v_{t}\in V_{k} $ for some $ 1\leq j \leq k \leq n $ and some $ 1\leq i \leq t \leq m, $ then $ c_{ik}\nsucc_{t} V_{p} $ for each $ 1\leq p \leq \ell $, a contradiction. If $|V\cap (V_{1}\cup V_{2}\cup \cdots \cup V_{n})|=0$, then $|V\cap (V_{n+1}\cup \cdots \cup V_{\ell})|\geq m$, and so $\ell \geq m+n$. In the other case, we may assume $  V\cap (V_{1}\cup V_{2}\cup \cdots \cup V_{n})  =\{v_{1}\}$ and $ v_{1}\in V_{1}$. Let $ v_{i} \in V_{n+i-1} $ for $ 2\leq i \leq m$. It can be easily seen that $f(c_{21})\geq m+n$ when $m=2$, and $f(c_{11})\geq m+n$ when $m>2$. So $ \ell \geq m+n$. For $m=2$, we consider $f=(V_{1}, V_{2},\cdots, V_{n+2}) $ where $ V_{i}=\{u_{i}\} $  for $ 1\leq i \leq n $, $V_{n+1}=\{v_{1}\} \cup \{c_{2i} ~|~1\leq i \leq n \} $ and $ V_{n+2}=\{v_{2}\}\cup \{c_{1i} ~|~1\leq i \leq n \} $, while for $m>2$ we consider $ f=(V_{1}, V_{2},\cdots, V_{m+n}) $ where $ V_{1}=\{v_{1}, u_{1}\} $, $ V_{i}=\{u_{i}\}$  for $2\leq i \leq n $, $V_{n+i-1}=\{v_{i}\}$ for $2\leq i \leq m $, and $ V_{m+n}=\{c_{ij}~|~ 1\leq i \leq m, 1\leq j \leq n\}$. Since $f$ is a TDC of $C(K_{m,n})$, we obtain $ \chi^{t}_{d}(C(K_{m,n}))=m+n$. 

\vspace{0.2cm}

In Figure \ref{fi:proofexample2}, $(\{v_1,u_1\},\{v_2\},\{v_3\},\{u_2\},\{u_3\},\{u_4\},\{u_5\},\{c_{1i},c_{2i},c_{3i}~|~1\leq i \leq 5\})$ is a min-TDC of $C(K_{3,5})$.
\end{proof}

\begin{figure}[ht]
\centerline{\includegraphics[width=5cm, height=5cm]{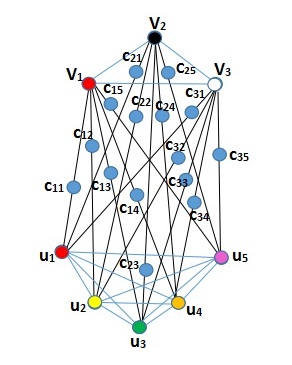}}
\vspace*{-0.4cm}
\caption{A min-TDC of $C(K_{3,5})$}\label{fi:proofexample2}
\end{figure}

\begin{prop}
\label{chi_d^t(C(K_{n_1, n_2, cdots ,n_w}))} For any complete $p$-partite graph $K_{n_1, n_2,\cdots,n_p}$ of order $ n\geq 4 $ in which $p\geq 3$ and $n_1\leq n_2 \leq \cdots \leq n_p$,
\begin{equation*}
\chi_d^t(C(K_{n_1, n_2,\cdots,n_p}))=\left\{
\begin{array}{ll}
n+1                   & \mbox{if } (n_1, n_2, \cdots, n_{p-1})=(2, 2,\cdots,2), \\
n+\lceil t_1/2\rceil  & \mbox{otherwise},
\end{array}
\right.
\end{equation*}
where $t_1=|\{~i~|~ n_{i}=1\}|$. 
\end{prop}

\begin{proof}
Let $G$ be the complete $p$-partite graph $K_{n_1, n_2,\cdots,n_p} $ of order $ n\geq 4 $ in which $n_1\leq n_2 \cdots \leq n_p$, $p\geq 3$ and $X_1\cup \cdots \cup  X_p$ is the partition of $V(G)=\{~v_i~|~1\leq i \leq n\}$ to the maximal independent sets $X_1$, $\cdots $, $X_p$ which have respectivly the cardinalities $n_1$, $\cdots $, $n_p$. Let $f=(V_{1}, V_{2},\cdots,V_{\ell} ) $ be a min-TDC of $ C(G)$. Similar to the proof of Proposition \ref{chi_d ^t(C(G))=n+ lceil n/2 rceil}, $ |\{~i~| ~V_i \subseteq V(G) \mbox{ and } |V_{i}|=1 \}|\geq n-2$, and so $ |\{f(v_{i})~|~ 1\leq i \leq n\}|\geq n-1$. Now $ \{f(v_{i})~|~ 1\leq i \leq n\}\cap\{f(c_{ij})~|~ c_{ij} \in \mathcal{C}\}=\emptyset $ implies $ \ell\geq n$. 

\vspace{0.2cm}
\textbf{Case 1.} $(n_1, n_2, \cdots, n_{p-1})=(2, 2,\cdots,2)$. Then $\ell=n$ implies $|\{~f(c_{ij})~|~ v_iv_j\in E(G)  \}|=1$ and $| \{~f(v_{i})~|~ v_i\in V(G) \}|=n-1$. Hence there exist two vertices $v\in X_i$ and $v' \in X_j$ for some $i\neq j$ such that $f(v)=f(v')$, and so there exists a vertex $v''$ in $X_i \cup X_j$ such that $v'' \nsucc_{t} V_{k} $ for each $ 1\leq k \leq n $, a contradiction. Thus $ \ell\geq n+1$. On the other hand, by the assumptions $X_1=\{v_1,v_2\}$ and $ v_{n-1},v_{n}\in X_p$, since  $g_0=(V_{1}, V_{2},\cdots,V_{n+1}) $ is a TDC of $ C(G) $ where  \[ V_{1}=\{v_{1},v_{n}\},~~ V_{i}=\{v_{i }\}~ \mbox{for}~ 2\leq i \leq n-1,~   V_{n}=\{c_{2 ( n-1)}\},~  V_{n+1}=\mathcal{C}-\{c_{2 ( n-1)}\},
\] 
we obtain $ \ell=n+1$.

\vspace{0.2cm}
\textbf{Case 2.} $(n_1, n_2, \cdots, n_{p-1})\neq (2, 2,\cdots,2)$. Since there is nothing to prove when $t_1=0$, we assume $t_1\neq 0$. Then, as we saw in the proof of Proposition \ref{chi_d ^t(C(G))=n+ lceil n/2 rceil}, since the central of the subgraph of $G$ induced by $X_1\cup \cdots \cup  X_{t_{1}}$ is isomorphic to the central of the complete graph $K_{t_{1}}$, we have $\{f(c_{ij})~|~ 1\leq i< j\leq t_{1}\} \cap \{f(v_{i})~|~ v_i\in X_1\cup \cdots \cup  X_{t_{1}} \}=\emptyset$, and so $\ell \geq n+\lceil t_1/2\rceil$. 

\vspace{0.2cm}
Now to complete the proof it is sufficient to give a TDC of $C(G)$ with minimum color classes. Let $t_1=0$. Then $n_p\geq n_{p-1}\geq 3$. Now by assumptions $v_1\in X_{p-1}$ and $v_2\in X_{p}$, since  $g_1=(V_{1}, V_{2},\cdots,V_{n}) $ is a TDC of $ C(G) $ where 
\[
V_{1}=\{v_{1},v_{2}\},~ V_{i}=\{v_{i }\} \mbox{ for } 3\leq i \leq n, ~ V_{n}=\mathcal{C},
\]
we obtain $ \ell=n$. 

\vspace{0.2cm}
For $ t_1\geq 2$, by assumptions $X_i=\{v_{i}\}$ for $1\leq i \leq t_1$, since $ g_2=(V_{1}, V_{2},\cdots,V_{n+\lceil t_1/2 \rceil})$ is a TDC of $ C(G) $ where
\[
V_{1}=\{v_{1},v_{2}\},  ~V_{i}=\{v_{i+1}\} \mbox{ for } 2\leq i \leq n-1,
\]
\[
V_{n+i}=\{c_{(2i+1) ( 2i+2)}\} \mbox{ for } 0\leq i \leq \lceil t_1/2\rceil-1,
\]
\[
V_{n+\lceil t_1/2 \rceil}= \mathcal{C}- (V_1\cup \cdots \cup V_{n+\lceil t_1/2 \rceil -1}),
\]
we obtain $ \ell=n+\lceil t_1/2\rceil$. 

\vspace{0.2cm}

If $ t_1=1$ and $|\{~i~|~n_i\geq 3\}|\geq 1$, then by assumption $v_2\in X_p$ the function $g_2$ will be again a TDC of $ C(G) $ and so $ \ell=n+\lceil t_1/2\rceil$. In the last case $G=K_{1,2,\cdots ,2}$, by assumptions $X_1=\{v_1\}$, $X_2=\{v_2,v_3\}$, since the coloring function $ g_3=(V_{1}, V_{2},\cdots,V_{n+\lceil t_1/2 \rceil})$ is a TDC of $ C(G) $ where
\[
V_{1}=\{v_{1},v_{2}\},  ~V_{i}=\{v_{i+1}\} \mbox{ for } 2\leq i \leq n-1,
\]
\[
V_{n}=\{c_{13}\}, \mbox{ and } V_{n+1}= \mathcal{C}- \{c_{13}\},
\]
we obtain $ \ell=n+\lceil t_1/2\rceil$. 
\vspace{0.2cm}

Figure \ref{fi:proofexample4} shows the central of $K_{3,3,3}$ with the partition $X_1\cup X_2 \cup X_3$ of its vertex set to the independent sets $X_{1}=\{v_{1}, v_{2}, v_{3}\}$, $X_{2}=\{v_{4}, v_{5}, v_{6}\}$ and $X_{3}=\{v_{7}, v_{8}, v_{9}\}$ and the min-TDS $(\{v_1,v_9\},\{v_2\},\{v_3\},\{v_4\},\{v_5\},\{v_6\},\{v_7\},\{v_8\},\mathcal{C})$ of $C(K_{3,3,3})$.

\begin{figure}[ht]
\centerline{\includegraphics[width=5.5cm, height=4.5cm]{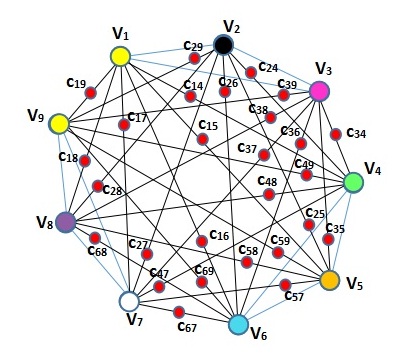}}
\vspace*{-0.3cm}
\caption{A min-TDC of $C(K_{3,3,3})$}\label{fi:proofexample4}
\end{figure}

\end{proof}

In the last proposition of this section we consider the double star graphs which are multipartite graphs but not complete. We recall that the \emph{double star} $S_{1,n,n}$ is a graph with the vertex set $\{ v_{0}, v_{1}, v_{2},\cdots, v_{n}, v_{n+1},\cdots,v_{2n} \}$ in which for $1\leq i \leq n$ every vertex $ v_{i}$ is adjacent to the two vertices $ v_{n+i} $ and $ v_{0}$ \cite{MNV}. 

\begin{prop}
\label{chi_d ^t (C(S_{1,n,n}))}
For any integer $n\geq 1$, $ \chi^{t}_{d}(C(S_{1,n,n}))=n+3$.
\end{prop}

\begin{proof}
Let $S_{1,n,n}$ be a double star graph with the vertex set $V=\{v_i~|~0\leq i \leq 2n\}$ and the edge set $E=\{v_0v_i,v_iv_{n+i}~|~1\leq i \leq n\}$. Then $V(C(S_{1,n,n}))=V\cup \mathcal{C}$ where $\mathcal{C}=\{c_{0i},c_{i(n+i)}~|~1\leq i \leq n\}$. Let $ f=(V_{1}, V_{2},\cdots,V_{\ell})$ be a TDC of $C(S_{1,n,n})$. Since the subgraph induced by $\{v_{n+i}~|~1\leq i \leq n\}\cup \{v_0\}$ is isomorphic to a complete graph of order $n+1$, we have $\chi^{t}_{d}(C(S_{1,n,n}))\geq n+1$. Without loss of generality, we may assume $v_{n+i}\in V_i$ for each $1\leq i \leq n$. Since $\ell=n+1$ implies $v_i\in V_i$ for each $1\leq i \leq n$, and so $ v_{0}\nsucc_{t} V_{i} $ for each $1\leq i \leq n$, we may assume $\ell\geq n+2$. Let $\ell=n+2$. If $ V_{i}=\{v_{n+i}\} $ for some $ 1\leq i \leq n$, then $ v_{i}\in V_{n+1} \cup V_{n+2}$, and so $c_{i(n+i)}\notin V_{1}\cup \cdots \cup V_{n+2}$, a contradiction. Therefore $ V_{i}=\{ v_{i}, v_{n+i}\}$ for $1\leq i \leq n$, and so $c_{i(n+i)} \in V_{n+2} $ and $c_{0i} \in V_{n+1} \cup V_{n+2}$ for each $ 1\leq i \leq n$. Since $c_{0i}\nsucc_{t} V_{j} $ for $1\leq j \leq n$, and $c_{i(n+i)} \in V_{n+2}$, we conclude $c_{0i}\succ_{t} V_{n+1}$, which implies $V_{n+1}=\{v_{0}\}$. Hence $ c_{0i}\in V_{n+2} $ for $1\leq i \leq n$, and so $v_{0}\nsucc_{t} V_{i} $ for each $ i$, a contradiction. Therefore $\ell \geq n+3$. Now since $ f=(V_{1},\cdots,V_{n+3})$ is a  TDC  of $C(S_{1,n,n})$ with $ n+3 $ color classes where $ V_{i}=\{v_{i}, v_{n+i}\}$ for $ 1\leq i \leq n$, $ V_{n+1}=\{ v_{0}\}$, $V_{n+2}=\{c_{i(n+i)}~|~1\leq i \leq n\}$, $ V_{n+3}=\{c_{0i}~|~ 1\leq i \leq n \}$, we obtain $ \chi^{t}_{d}(C(S_{1,n,n}))=n+3$. 
\vspace{0.2cm}

In Figure \ref{fi:proofexample1}, the coloring function $(\{v_0\},\{v_2,v_5\},\{c_{14},c_{25},c_{36}\},\{c_{01},c_{02},c_{03}\},\{v_1,v_4\},\{v_3,v_6\})$ is a min-TDC of $C(S_{1,3,3})$.

\begin{figure}[ht]
\centerline{\includegraphics[width=4cm, height=4cm]{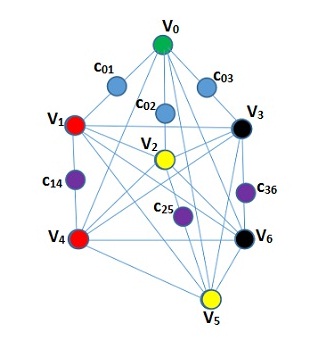}}
\vspace*{-0.4cm}
\caption{A min-TDC of $C(S_{1,3,3})$}\label{fi:proofexample1}
\end{figure}
\end{proof}


\section{Problems}

Finally we end our discussion with some problems and questions for further researchs.

\begin{prob}
Characterize graphs $ G $ satisfies $\chi_{d}^{t}(C(G))=\chi_{d}^{t}(G)$.
\end{prob}

\begin{prob}
For $t<n$ find connected graphs $G$ of order $n\geq 2$ with a longest path of order $t$ such that $\chi_d ^t(C(G))=n+ \lceil t/2 \rceil$.
\end{prob}

\begin{ques}
Whether for any connected graph $G$ of order at least $2$, $ \chi_d ^t(C(G))\geq \chi_d ^t(G)$?  
\end{ques}

\begin{ques}
Whether for any connected graph $G$ of order at least $3$, $ \chi_d ^t(C(G)) \neq \chi (C(G))$?
\end{ques}



\begin{thebibliography}{20}


\bibitem{GJ} M. R. Garey, D. S. Johnson, \textit{Computers and Intractability}, W. H. Freeman and Co., 1978.

\bibitem{hhs1} T. W. Haynes, S. T. Hedetniemi, P. J. Slater (eds). \textit{Fundamentals Domination in Graphs}, Marcel Dekker, Inc. New York, 1998.

\bibitem{hhs2} T. W. Haynes, S. T. Hedetniemi, P. J. Slater (eds), \textit{Domination in Graphs: Advanced Topics}, Marcel Dekker, Inc. New York, 1998.

\bibitem{Hen2015} M. A. Henning, Total dominator colorings and total domination in graphs, \emph{Graphs and Combinatorics}, 31 (2015) 953--974.

\bibitem{HeYe13} M. A. Henning, A. Yeo, \emph{Total domination in graphs} (Springer Monographs in Mathematics) (2013) ISBN: 978-1-4614-6524-9 (Print) 978-1-4614-6525-6 (Online).

\bibitem{Kaz-Par}  P. Jalilolghadr, A. P. Kazemi, A. Khodkar, Total dominator coloring of the circulant graphs $C_n(a,b)$, {\em manuscript}.

\bibitem{Kaz2015}  A. P. Kazemi, Total dominator chromatic number of a graph, {\em Transactions on Combinatorics}, 4(2) (2015), 57--68.

\bibitem{Kaz2014}  A. P. Kazemi, Total dominator coloring in product graphs, {\em Utilitas Mathematica}, 94 (2014) 329--345.

\bibitem{Kaz2016}  A. P. Kazemi, Total dominator chromatic number of Mycieleskian graphs, {\em Utilatas Mathematica}, 103 (2017) 129-137.

\bibitem{Nordhaus}  E.A. Nordhaus and j.w. Gaddum, On compelementary graphs {\em Amer. Math. Monthly}, 63 (1956) 175-177.


\bibitem{MNV}M. Venkatachalam, N. Mohanapriya, J. V. Vivin, Star coloring on double star graph families, \emph{Journal of Modern Mathematics and Statistics}, 5(1) (2011) 33--36.

\bibitem{Vv}  J. V. Vernold, Harmonious coloring of total graphs, $n$-leaf, central graphs and circumdetic graphs, Ph.D Thesis, Bharathiar University, Coimbatore, India (2007).

\bibitem{West} {D. B. West,} \textit{Introduction to Graph Theory}, 2nd ed, prentice hall, USA, (2001).
\end{thebibliography}
\end{document}